\documentclass[leqno]{amsart}
\usepackage{amsthm}
\usepackage{color}
\usepackage{amssymb}
\usepackage{amsmath,amsfonts,enumerate}
\usepackage{amsthm}

\numberwithin{equation}{section}

\newtheorem{thm}{Theorem}
\newtheorem{lem}[thm]{Lemma}

\newtheorem{rmk}[thm]{Remark}

\begin{document}

\title{Noetherian solvability of an operator singular integral equation with a Carleman shift in fractional spaces}

\author{N.K. Bliev}
\address{Institute of Mathematics and Mathematical Modeling, 050010 Almaty, Kazakhstan}
\email{bliyev.nazarbay@mail.ru}
\author{K.S. Tulenov}
\address{Al-Farabi Kazakh National University, 050040 Almaty, Kazakhstan; Institute of Mathematics and Mathematical Modeling, 050010 Almaty, Kazakhstan}
\email{tulenov@math.kz}

\date{}
\begin{abstract} In this paper, we obtain conditions of Noetherian solvability and the Index formula for a singular integral equation with a Cauchy kernel and a Carleman shift in Besov space, which
is embedded into the space of continuous functions on a closed Lyapunov contour, but not into the class of functions satisfying H\"{o}lder condition.
\end{abstract}

\maketitle

\section{Introduction}
Let $\Gamma$ be a closed Lyapunov contour from the class $C_\nu ^1$, $\frac{2}{p} - 1 < \nu  \le 1$, $1 < p < 2$. Let $1 < p < 2$ and  $r = \frac{1}{p}.$ For brevity, we denote Besov spaces $B_{p,1}^r\left( \Gamma  \right)$ by $B\left( \Gamma  \right).$ In $B\left( \Gamma  \right)$, consider the singular integral equation (SIE)
\begin{equation}\label{eq: 1}
\begin{split}
M\varphi  \equiv& a\left( t \right)\varphi \left( t \right) + b\left( t \right)\varphi \left[ {\alpha \left( t \right)} \right] + \frac{{c\left( t \right)}}{{\pi i}}\int\limits_\Gamma  {\frac{{\varphi \left( \tau  \right)}}{{\tau  - t}}d\tau }  + \frac{{d\left( t \right)}}{{\pi i}}\int\limits_\Gamma  {\frac{{\varphi \left( \tau  \right)}}{{\tau  - \alpha \left( t \right)}}d\tau } \\
 &+ \int\limits_\Gamma  {K\left( {t,\tau } \right)\varphi \left( \tau  \right)d\tau }  = g\left( t \right),
\end{split}
\end{equation}
where $a(\cdot)$, $b(\cdot)$, $c(\cdot)$, $d(\cdot)$ and $g(\cdot)$ are functions from the space $B(\Gamma)$, $\alpha(\cdot)$ is a Carleman shift: homeomorphically maps $\Gamma$ into itself with preservation or change of orientation on $\Gamma$, and satisfies condition
\begin{equation}\label{eq: 2}
\alpha \left[ {\alpha \left( t \right)} \right] = t.
\end{equation}
Suppose that there exist derivative $\alpha '\left( t \right)$ of $\alpha \left( t \right)$ which belongs to the class ${H_\mu }\left( \Gamma  \right)$ ($0 < \mu  \le 1$) of functions defined on $\Gamma$ satisfying the H\"{o}lder condition with exponent $\mu$. The spaces ${H_\mu }\left( \Gamma  \right)$ ($0 < \mu  \le 1$) are Banach space with the norm
 $$\|\varphi\|_{H_{\mu}( \Gamma )}:=\max_{t\in \Gamma}|\varphi(t)|+\sup_{t,\tau\in \Gamma}\frac{|\varphi(\tau)-\varphi(t)|}{|\tau-t|^{\mu}}$$
 (see \cite[Chapter I.2, p. 16]{KL}).
 Let the kernel $K\left( {t,\tau } \right)$ in \eqref{eq: 1} has such a weak singularity that the corresponding integral operator is completely continuous in $B(\Gamma)$.
It is well-known that the space $B(\Gamma)$ is embedded into the space $C(\Gamma)$ of continuous functions on $\Gamma$, but not into the space ${H_\mu }\left( \Gamma  \right)$, for any $0 < \mu  \le 1$ (see \cite{AIN}). Moreover, $B(\Gamma)$ is a commutative Banach algebra with ordinary operations of addition and multiplication \cite{B1, B2}. It was proved in \cite[pp. 45-47]{B1} that operator $M$ is bounded on  $B(\Gamma).$ The equation (\ref{eq: 1}) was studied in the spaces ${H_\mu }\left( \Gamma  \right)$, $0 < \mu  \le 1$, and ${L_p}\left(\Gamma\right)$,
$1 < p < \infty $, in \cite{L}. We also refer the interested reader to \cite{KL} and the references therein.
Boundary value problems with operators with complex conjugate and a Carleman shift and the corresponding singular integral equations (SIE) including such operators are one of the important part of the theory of analytic functions of complex variables. Such problems have significant applications and they are played crucial role, for example, in the theory of boundary value problems for differential operators \cite{BD,B4}. Boundary value problems with shift were investigated by D. Hilbert in the beginning of the 20th century. Singular integral equations with a Cauchy kernel and boundary value problems were studied by many Mathematicians and Mechanics. For more details on this direction we refer the reader to \cite{M, PS}. An application of Besov spaces possessing limit embedding theorems allows to weaken the condition from the density of H\"{o}lder continuous integrals by replacing it just continuity in terms of Besov spaces. This direction was developed by the first author \cite{B1,B2}. The theory of SIE with a Cauchy kernel and the corresponding boundary value problems for piecewise analytic functions have been investigated sufficiently. However, almost all works in this direction were studied in the spaces $H_{\mu},$ $0<\mu\leq1,$ and $L_{p}$, $p>1$ \cite{KL,L}.
 In this work, we obtain conditions on Noetherian solvability of the singular integral equation \eqref{eq: 1} with a Cauchy kernel and a Carleman shift in Besov spaces $B_{p,1}^{r}(\Gamma)$, $1<p<2,$ $r=\frac{1}{p},$ which is embedded into the class $C(\Gamma)$ of continuous functions on a closed Lyapunov contour $\Gamma$, but not the class $H_{\mu}(\Gamma),$ $0<\mu\leq 1,$ of functions satisfying H\"{o}lder condition with exponent $\mu$. We also obtain the Index formula for this singular integral equation. A corresponding characteristic equation is equivalent to a boundary value problem for piecewise analytic functions vanishing at infinity.

\section{Preliminaries}

\subsection{${B_2}\left( \Gamma  \right)$  and  ${B_{2\times 2}(\Gamma)}$ spaces}
Let ${B_2}\left( \Gamma  \right)$ be a set of all 2-dimensional vectors with components from ${B}\left( \Gamma  \right)$ and $B_{2\times2}(\Gamma)$  be a set of all $2\times2$ square matrices with elements from ${B}\left( \Gamma  \right)$. The set $B_2(\Gamma)$ can be supplied with the norm by taking as the norm of the vector $X = \left( {{x_1},{x_2}} \right)$ and considering the sum of the norms of the individual components:
$$\left\| X \right\|_{B_2(\Gamma)}: = {\left\| {{x_1}} \right\|_{B\left( \Gamma  \right)}} + {\left\| {{x_2}} \right\|_{B\left( \Gamma  \right)}}.$$

In this case, the norm of the matrix $A = \left\{ {{a_{kj}}} \right\}_1^2 \in {B_{2\times2}}\left( \Gamma  \right)$  can be defined, for example, by
$$\left\| A \right\|_{B_{2\times2}(\Gamma) } := 2\mathop {\max }\limits_{j,k} {\left\| {{a_{jk}}} \right\|_{B\left( \Gamma  \right)}}.$$

Then the space ${B_{2\times 2}(\Gamma)}$ with this norm will also be a Banach algebra \cite{B3}.

Let $L(B)$ be the set of all continuous linear operators in $B(\Gamma)$. Then each operator $A \in L\left( {{B_2}} \right)$ can be represented as a two-dimensional matrix $A = \left\{ {{a_{j k}}} \right\}$, where ${a_{jk}} \in L\left( B \right)$.  Moreover, the matrix operator $A$ is completely continuous if and only if all operators multiplied by ${a_{jk}}$ in $B(\Gamma)$ are completely continuous.

\subsection{Noether operators and their $d$-characteristics}
Set integers
$$\alpha \left( L \right) = \dim \ker L,\,\,\beta \left( L \right) = \dim \, \text{coker} L.$$

An operator $L$ is said to have a finite $d$ - characteristics or a finite index (we will denote it by $Ind$) if both $\alpha(L)$ and $\beta(L)$ are finite. A linear closed normally solvable (in the sense of Hausdorff \cite{L}) operator $L$ is called a Noetherian operator or $F$-operator, if its $d$-characteristic or index is finite, and semi-Noetherian, if at least one of the numbers $\alpha(L)$ and $\beta(L)$ are finite. In the case of semi-Noetherian, we will distinguish between $F_+$ - operators $(\alpha \left( L \right) < \infty )$ and $F_-$ - operators $(\beta \left( L \right) < \infty )$.

Let $G$ be a region of the complex plane $E$ bounded by a closed Lyapunov contour $\Gamma$ and let $G_{+}:=G$, $G_{-}:=E-\overline {G_{+}}$. Throughout this paper, we assume that $0 \in G_{ + }$ and $z = \infty  \in G_{-}.$

A right factorization of non-singular for everywhere on $\Gamma$ matrices $A(t) \in {B_{2\times2}}(\Gamma)$ is called the following its representation

\begin{equation}\label{eq: 11}
A\left( t \right) = {A_-}\left( t \right)\cdot D \left( t \right)\cdot{A_ + }\left( t \right),\, t \in \Gamma ,
\end{equation}
with a diagonal matrix $D(t)$ as $D\left( t \right) = \left\{ {{t^{jk}}{\delta _{jk}}} \right\}_1^2$, $t \in \Gamma $. Here,
 ${k_1} \ge {k_2}$ are some numbers which are uniquely determined by the matrices $A(t),$  $A_{\pm}$ are square matrices of the second order, which can be continued analytically in the domains $G_{\pm}$ and continuous in $\overline{G_ +} $, and satisfying
$$det{A_ + }\left( z \right) \ne 0, \,\ z \in \overline {G_{\pm}} ,\,det{A_ - }\left( z \right) \ne 0 , \,\ z \in \overline {G_{-}}.$$

The factorization of the matrix $A(t)$, which is obtained by swapping the multipliers $A_{±}(t)$ in (\ref{eq: 11}), is called the left factorization. Obviously, each right (left) factorization of the matrix $A(t),$ as well as the inverse matrix $A^{-1}(t),$ generates right (left) factorization of transposed matrix $A^{T}(t).$
 Clearly, $B(\Gamma)$ is a disintegrating $R$-algebra \cite[Subsection $4^0$]{B3}, therefore, every non-singular matrix $A\left( t \right) \in {B_{2\times2}}\left( \Gamma \right)$ admits a right (and left) factorization \cite[Corollary VII. 2.1, p. 309]{PS}.

\section{Noetherian solvability of SIE and its index formula}

In this section we prove a Noetherian solvability of \eqref{eq: 1} and obtain its Index formula.
First we need some preparation.
Let us consider the singular integral equation (SIE) (\ref{eq: 1}) in Besov space $B(\Gamma).$
Characteristic equation ($K( {t,\tau }) = 0$) (\ref{eq: 1}) is equivalent to a boundary value problem of finding a piecewise analytic functions $\left\{ {{\Phi^ + }\left( z \right),{\Phi^ - }\left( z \right)} \right\}$ vanishing at infinity,

\begin{equation}\label{eq: 3}
{a_1}\left( t \right){\Phi ^ + }\left( t \right) + {b_1}\left( t \right){\Phi ^ + }\left[ {\alpha \left( t \right)} \right] + {c_1}\left( t \right){\Phi ^ - }\left( t \right) + {d_1}\left( t \right){\Phi ^ - }\left[ {\alpha \left( t \right)} \right] = g\left( t \right),
\end{equation}
where
\begin{equation}
\begin{array}{l}
{a_1}\left( t \right) = a\left( t \right) + c\left( t\right), \quad {c_1}\left( t \right) = c\left( t \right) - a\left( t \right),\\
{b_1}\left( t \right) = b\left( t \right) + d\left( t\right), \quad {d_1}\left( t \right) = d\left( t \right) - b\left( t \right).
\end{array}
\label{eq: 4}
\end{equation}
Following \cite{L}, we construct an auxiliary system of two singular integral equations without shift which is called {\it corresponding system} to the equation \eqref{eq: 1}. For this we substitute the independent variable $t$ into $\alpha(t)$ in the equation \eqref{eq: 1}, and associate such obtained equation with the initial equation \eqref{eq: 1}.
Setting new unknown functions
$${\rho _1}\left( t \right) = \varphi \left( t \right),\quad {\rho _2}\left( t \right) = \varphi \left[ {\alpha \left( t \right)} \right],$$
we obtain a corresponding system of two SIEs with a Cauchy kernel  relatively unknown vector $\rho \left( t \right) = \left\{ {{\rho _1}\left( t \right),{\rho _2}\left( t \right)} \right\}$

\begin{equation}\label{eq: 5}
\begin{split}
a(t){\rho _1}\left( t \right)& + b\left( t \right){\rho _2}\left( t \right) + \frac{{c\left( t \right)}}{{\pi i}}\int\limits_\Gamma  {\frac{{{\rho _1}\left( \tau  \right)}}{{\tau  - t}}d\tau }  \\
&+\frac{{\gamma d\left( t \right)}}{{\pi i}}\int\limits_\Gamma  {\frac{{\alpha '\left( \tau  \right){\rho _2}\left( \tau  \right)}}{{\alpha \left( \tau  \right) - \alpha \left( t \right)}}d\tau }  + \int\limits_\Gamma  {K\left( {t,\tau } \right){\rho _1}\left( \tau  \right)d\tau }  = g\left( t \right),\\
b\left[ {\alpha \left( t \right)} \right]{\rho _1}\left( t \right)& + a\left[ {\alpha \left( t \right)} \right]{\rho _2}\left( t \right) + \frac{{d\left[ {\alpha \left( t \right)} \right]}}{{\pi i}}\int\limits_\Gamma  {\frac{{{\rho _1}\left( \tau  \right)}}{{\tau  - t}}d\tau }  + \frac{{\gamma c\left[ {\alpha \left( t \right)} \right]}}{{\pi i}}\int\limits_\Gamma  {\frac{{\alpha '\left( \tau  \right){\rho _2}\left( \tau  \right)}}{{\alpha \left( \tau  \right) - \alpha \left( t \right)}}d\tau } \\
&+\gamma \int\limits_\Gamma  {K\left[ {\alpha \left( t \right),\alpha \left( \tau  \right)} \right]\alpha '\left( \tau  \right){\rho _2}\left( \tau  \right)d\tau }  = g\left[ {\alpha \left( t \right)} \right],
\end{split}
\end{equation}
where the coefficient $\gamma$ takes the value 1 or -1, if respectively $\alpha(t)$ maps $\Gamma$ into itself with preservation or change of orientation on $\Gamma$.  An operator,  applying of which to the vector $\left( {{\rho _1},{\rho _2}} \right)$ gives the left hand side of (\ref{eq: 5}), can be written as

\begin{equation}\label{eq: 5a}
L \equiv P\left( t \right)I + Q\left( t \right)S + {D_1},
\end{equation}
where matrices $P\left( t \right)$ and $Q\left( t \right)$ have forms
\begin{equation}\label{PQ}
\begin{array}{l}
P\left( t \right) = \left( {\begin{array}{*{20}{c}}
{a\left( t \right)}&{b\left( t \right)}\\
{b\left[ {\alpha \left( t \right)} \right]}&{a\left[ {\alpha \left( t \right)} \right]}
\end{array}} \right),\quad Q\left( t \right) = \left( {\begin{array}{*{20}{c}}
{c\left( t \right)}&{\gamma d\left( t \right)}\\
{d\left[ {\alpha \left( t \right)} \right]}&{\gamma c\left[ {\alpha \left( t \right)} \right]}
\end{array}} \right),\\
\end{array}
\end{equation}
 $D_{1}$ is a completely continuous operator, $I$ is an identity operator in $B_{2}(\Gamma),$ and $S$ is a scalar integral operator defined by the following formula
\begin{equation}\label{S oper}(S\varphi)(t)  = \frac{1}{{\pi i}}\int\limits_\Gamma  {\frac{{\varphi \left( \tau  \right)}}{{\tau  - t}}d\tau }, \quad t\in \Gamma.
\end{equation}
Note that for any $\varphi\in B(\Gamma)$ the above singular integral (see for more information \cite{P}) exists in the sense of principle value \cite{B1} (see also \cite{B3}).
Operator $L$  in (\ref{eq: 5a}) acts in the space $B(\Gamma)$ (for more details see \cite{B1}).
Moreover,
it belongs to $L(B_2)$ and can be written as

\begin{equation}\label{eq: 10}
L = C{P_1} + D{Q_1} + T,
\end{equation}
where $C$ and $D$ are multiplication operators in $B_2(\Gamma)$ corresponding matrix functions
$$
C\left( t \right) =P\left( t \right) + Q\left( t \right) \quad \text{and} \quad D\left( t \right) = P\left( t \right) - Q\left( t \right),
$$
where $P(t)$ and $Q(t)$ defined in \eqref{PQ},
$$
{S_1} = \left\{ {S{\delta _{jk}}} \right\}_1^2,\,\,{P_1} = \frac{1}{2}\left( {I + {S_1}} \right),\,\,{Q_1} = \frac{1}{2}\left( {I - {S_1}} \right),
$$
and $S$ is a scalar integral operator defined in \eqref{S oper}, $I$ is an identity operator in $B_{2}(\Gamma)$,  $\delta_{jk}$ is the Kronecker symbol.
It was proved in \cite{B1} that $S$ is bounded in $B(\Gamma).$
Operators as $L$ generates an algebra \cite{B2}.

We introduce additional SIE with a Carleman shift corresponding to (\ref{eq: 1})

\begin{equation}\label{eq: 6}
\begin{split}
K\chi  &\equiv a\left( t \right){\chi}\left( t \right) - b\left( t \right){\chi}\left[ {\alpha \left( t \right)} \right] + \frac{{c\left( t \right)}}{{\pi i}}\int\limits_\Gamma  {\frac{{{\chi}\left( \tau  \right)}}{{\tau  - t}}d\tau }  - \frac{{d\left( t \right)}}{{\pi i}}\int\limits_\Gamma  {\frac{{{\chi}\left( \tau  \right)}}{{\tau  - \alpha \left( t \right)}}d\tau } \\
&+\int\limits_\Gamma  {K\left( {t,\tau } \right){\chi}\left( \tau  \right)d\tau }  = 0.
\end{split}
\end{equation}
It is easy to see that if ${\rho _1}\left( t \right) = {\chi}\left( t \right)$, ${\rho _2}\left( t \right) = - {\chi}\left(\alpha t \right)$, then using the procedure described above, you can obtain the corresponding system (\ref{eq: 5}) from (\ref{eq: 6}).

We also consider two SIEs with a Carleman shift, one of which is union with (\ref{eq: 1}), and the other is union with the accompanying equation (\ref{eq: 6}). Union operators are understood in the sense of the following identities
$$\int\limits_\Gamma  {\left( {M\varphi } \right)\left( t \right)\psi \left( t \right)dt}  = \int\limits_\Gamma  {\varphi \left( t \right)\left( {M'\psi } \right)\left( t \right)dt}.$$
The mentioned union equations have the form
\begin{equation}\label{eq: 7}
\begin{split}
M'\psi  &\equiv a\left( t \right)\psi \left( t \right) + \gamma \alpha '\left( t \right)b\left[ {\alpha \left( t \right)} \right]\psi \left[ {\alpha \left( t \right)} \right] - \frac{1}{{\pi i}}\int\limits_\Gamma  {\frac{{c\left( \tau  \right)\psi \left( \tau  \right)}}{{\tau  - t}}d\tau } \\
 &- \frac{\gamma }{{\pi i}}\int\limits_\Gamma  {\frac{{d\left[ {\alpha \left( \tau  \right)} \right]\alpha '\left( \tau  \right)\psi \left[ {\alpha \left( \tau  \right)} \right]}}{{\tau  - t}}d\tau }  + \int\limits_\Gamma  {m\left( {\tau ,t} \right)\psi \left( \tau  \right)d\tau }  = 0,
\end{split}
\end{equation}
and
\begin{equation}\label{eq: 8}
\begin{split}
K'\omega  &\equiv a\left( t \right)\omega \left( t \right) + \gamma \alpha '\left( t \right)b\left[ {\alpha \left( t \right)} \right]\omega \left[ {\alpha \left( t \right)} \right] - \frac{1}{{\pi i}}\int\limits_\Gamma  {\frac{{c\left( \tau  \right)\omega \left( \tau  \right)}}{{\tau  - t}}d\tau } \\
& + \frac{\gamma }{{\pi i}}\int\limits_\Gamma  {\frac{{d\left[ {\alpha \left( \tau  \right)} \right]\alpha '\left( \tau  \right)\omega \left[ {\alpha \left( \tau  \right)} \right]}}{{\tau  - t}}d\tau }  + \int\limits_\Gamma  {m\left( {\tau ,t} \right)\omega \left( \tau  \right)d\tau }  = 0.
\end{split}
\end{equation}

It is clear that equation (\ref{eq: 8}) is accompanying to the union equation (\ref{eq: 7}). The system of SIE without shift corresponding to the equations (\ref{eq: 8}) and (\ref{eq: 7}) have forms

\begin{equation}\label{eq: 9}
\begin{split}
a(t){\omega _1}\left( t \right)&+ b\left[ {\alpha \left( t \right)} \right]{\omega _2}\left( t \right) - \frac{1}{{\pi i}}\int\limits_\Gamma  {\frac{{c\left( \tau  \right){\omega _1}\left( \tau  \right)}}{{\tau  - t}}d\tau } \\
&- \frac{1}{{\pi i}}\int\limits_\Gamma  {\frac{{d\left[ {\alpha \left( \tau  \right)} \right]{\omega _2}\left( \tau  \right)}}{{\tau  - t}}d\tau }  + \int\limits_\Gamma  {m\left( {\tau ,t} \right){\omega _1}\left( \tau  \right)d\tau }  = 0,\\
b\left( t \right){\omega _1}\left( t \right)& + a\left[ {\alpha \left( t \right)} \right]{\omega _2}\left( t \right) - \frac{{\gamma \alpha '\left( t \right)}}{{\pi i}}\int\limits_\Gamma  {\frac{{d\left( \tau  \right){\omega _1}\left( \tau  \right)}}{{\alpha \left( \tau  \right) - \alpha \left( t \right)}}d\tau }  - \frac{{\gamma \alpha '\left( t \right)}}{{\pi i}}\int\limits_\Gamma  {\frac{{c\left[ {\alpha \left( \tau  \right)} \right]{\omega _2}\left( \tau  \right)}}{{\alpha \left( \tau  \right) - \alpha \left( t \right)}}d\tau } \\
 &+ \gamma \alpha '\left( t \right)\int\limits_\Gamma  {m\left[ {\alpha \left( \tau  \right),\alpha \left( t \right)} \right]{\omega _2}\left( \tau  \right)d\tau }  = 0.
\end{split}
\end{equation}
It is easy to verify that the system (\ref{eq: 9}) is union with the corresponding system of equations (\ref{eq: 5}).
Taking into account all the above information, we obtain the following results.

\begin{lem}\label{auxiliary lem1}
The number $l$ of linearly independent solutions of the homogeneous ($g\left( t \right) \equiv 0$) equation corresponding to the system of equations (\ref{eq: 5}) is equal to the sum of numbers $l_1$ and $l_2$ of the linearly independent solutions of this homogeneous equation (\ref{eq: 1}) and the accompanying equation (\ref{eq: 6}), respectively, i.e $l=l_1+l_2$.
\end{lem}
\begin{proof} It follows from the method of constructing corresponding system \eqref{eq: 5} that for every solution of homogeneous equation $g\left( t \right) \equiv 0$ is corresponded a certain solution of the homogeneous system \eqref{eq: 5}. Indeed, if $\varphi$ is a solution of the homogeneous equation \eqref{eq: 1}, then a vector $\rho(t)=\{\rho_{1}(t),\rho_{2}(t)\},$ where $\rho_{1}(t)=\varphi(t)$ and $\rho_{2}(t)=\varphi[\alpha(t)],$ becomes a solution of the homogeneous system \eqref{eq: 5}. Moreover, this solution satisfies the following condition
\begin{equation}\label{ro2}
\rho_{2}[\alpha(t)]=\rho_{1}(t).
\end{equation}
Similarly, if $\chi$ is a solution of the homogeneous equation \eqref{eq: 6}, then a vector $\rho(t)=\{\rho_{1}(t),\rho_{2}(t)\}$ with components $\rho_{1}(t)=\chi(t)$ and $\rho_{2}(t)=-\chi[\alpha(t)]$ is a solution of the homogeneous system \eqref{eq: 5} satisfying
\begin{equation}\label{-ro2}\rho_{2}[\alpha(t)]=-\rho_{1}(t).
\end{equation}
Futher, let $l$ be a number of linearly independent solutions of the homogeneous system of equations \eqref{eq: 5} and let $l_1$ and $l_2$ be numbers of the linearly independent solutions of this homogeneous equation (\ref{eq: 1}) and (\ref{eq: 6}), respectively. Then, by using \eqref{ro2} and \eqref{-ro2}, it is easy to see that $l=l_1+l_2.$ It can be proved by repeating argument in the proof of Lemma 6.1 from \cite[pp. 63-66]{L}.
\end{proof}

Let ${l^*}$ be a number of linearly independent solutions of the system (\ref{eq: 9}), allied with the corresponding system of equations (\ref{eq: 5}), and $l_1^*$ and $l_2^*$ be numbers of linearly independent solutions of equations (\ref{eq: 7}) and (\ref{eq: 8}), which are the union equations with (\ref{eq: 1}) and the accompanying equation (\ref{eq: 6}), respectively.

Similar to Lemma \ref{auxiliary lem1}, we obtain the following result.
\begin{lem}\label{auxiliary lem2}
Numbers $l^*$, $l_1^*$ and $l_2^*$ satisfy the equation $l^*=l_1^*+l_2^*$.
\end{lem}
\begin{proof} The proof of this lemma can be proved in a similar manner to Lemma \ref{auxiliary lem1}, if we take into account that it can be selected a complete system of linearly independent solutions of the union system of equations \eqref{eq: 9} so that $l_1^*$ solutions of this system satisfy
\begin{equation}\label{l1*}w_2(t)-\gamma\alpha'(t)w_1[\alpha(t)]=0,
\end{equation}
and other $l_2^*$ solutions satisfy
\begin{equation}\label{l2*}w_2(t)+\gamma\alpha'(t)w_1[\alpha(t)]=0.
\end{equation}
\end{proof}
\begin{lem}\label{auxiliary lem3}
Inhomogeneous SIE (\ref{eq: 1}) with a Carleman shift is solvable if and only if the corresponding inhomogeneous system of equations (\ref{eq: 5}) is solvable.
\end{lem}
\begin{proof} Suppose that the equation \eqref{eq: 1} is solvable and $\varphi$ is its solution. Then a vector
$\rho$ with components $\rho_{1}(t)=\varphi(t)$ and $\rho_{2}(t)=\varphi[\alpha(t)]$ is a solution of the system \eqref{eq: 5}.
Conversely, if a vector $\rho(t)=\{\rho_{1}(t),\rho_{2}(t)\}$ is a solution of the system \eqref{eq: 5}, then by an easy calculation it can be seen that the vector
$\rho(t)=\{\rho_{1}[\alpha(t)],\rho_{2}[\alpha(t)]\}$ is also a solution of the system \eqref{eq: 5}.
Therefore, a vector
$$\widetilde{\rho}(t)=\{\widetilde{\rho_{1}}(t),\widetilde{\rho_{2}}(t)\},$$
where $\widetilde{\rho_{1}}(t)=\frac{1}{2}(\rho_{1}(t)+\rho_{2}[\alpha(t)])$ and $\widetilde{\rho_{2}}(t)=\frac{1}{2}(\rho_{2}(t)+\rho_{1}[\alpha(t)]),$
is also a solution of the system \eqref{eq: 5}. This solution satisfies condition $\widetilde{\rho_{2}}[\alpha(t)]=\widetilde{\rho_{1}}(t).$
Consequently, the function $\varphi(t)=\rho_{1}(t)$ is a solution of such equation \eqref{eq: 5}, thereby completing the proof.

\end{proof}

\begin{rmk}\label{indm=indk} It can be easily checked (see \cite[pp. 68-69]{L}) that between operators $M$ in \eqref{eq: 1} and $K$ in \eqref{eq: 6} have the following connection. If a Carleman shift $\alpha(\cdot)$ preserves orientation on $\Gamma,$ then
\begin{equation}\label{UM}
uM-Ku=D,
\end{equation}
where $u(t)=\alpha(t)-t\neq 0$ (see \cite[p. 25]{L}) and $D$ is a completely continuous operator. If
$\alpha(\cdot)$ changes orientation on $\Gamma,$ then
\begin{equation}\label{SM}
SM-KS=D,
\end{equation}
where $S$ is the operator defined in \eqref{S oper} and $D$ is a completely continuous operator.
Since $u(t)\neq 0,$ it follows that $S$ is a Noetherian operator. Hence, by \eqref{UM} and \eqref{SM} we obtain that operators $M$ and $K$ is simultaneously Noetherian or not. In the case both are Noetherian operators, we have
$$Ind M=Ind K.$$
\end{rmk}

Now we can formulate the conditions for the Noetherity of the operator $L$ (the Noetherian solvability of the system (\ref{eq: 5})) in $L(B_2)$,  which is written in the form (\ref{eq: 5a}). The following result follows from Theorems 2 and 3 in \cite{B3} when $n=2$.
\begin{thm}\label{thm1}
Operator $L=CP_1+DQ_1+T$ defined in \eqref{eq: 10} is to be $F_+$ (or $F_-$)- operator if and only if
\begin{equation}\label{eq: 12}
\det C\left( t \right) \ne 0,\,\det D\left( t \right) \ne 0,\, t \in\Gamma .
\end{equation}
Conversely, if these conditions hold, then $R=C^{-1} P_1+D^{-1} Q_1$ is a two-sided regularizer of the operator $L$.
\end{thm}

\begin{rmk}
 If condition (\ref{eq: 12}) holds, then the matrix $D^{-1} (t)C(t)$ have a right factorization
$${D^{ - 1}}\left( t \right)C\left( t \right) = C_{-}\left( t \right)U\left( t \right)C_{+}\left( t \right),U\left( t \right) = \left\{ {{t^{{K_j}}}{\delta _{kj}}} \right\}.$$
\end{rmk}

\begin{thm}\label{thm2}
Let condition (\ref{eq: 12}) holds. If the characteristic equation
$$\left( {{P_1} + D{Q_1}} \right)\rho  = f, \quad f \in {B_{2}} (\Gamma),$$
is solvable, then its solution is given by the following formula
$$\rho  = {\left( {C{P_1} + D{Q_1}} \right)^{ - 1}}\left( f \right),$$
for which
$${\left( {{P_1} + D{Q_1}} \right)^{ - 1}} = \left( {C_ + ^{ - 1}{P_1} + {C_ - }{Q_1}} \right)\left( {{U^{ - 1}}{P_1} + {Q_1}} \right)C_ - ^{ - 1}{D^{ - 1}}$$
have left and right inverse operator to $CP_1+DQ_1$ depending on the $k = Ind\left( {D_1^{ - 1}{C_1}} \right) \ge 0$ or $k \le 0.$
The index of the Noetherian operator corresponding to the system \eqref{eq: 5} is calculated by the following formula
\begin{equation}
Ind L= \frac{1}{{2\pi }}{\left\{ {arg\frac{{\det D\left( t \right)}}{{\det C\left( t \right)}}} \right\}_\Gamma }.
\label{eq: 13}
\end{equation}
\end{thm}

Following \cite{L}, condition (\ref{eq: 12}) and formula (\ref{eq: 13}) can be rewritten according to the following cases: $\alpha(t)$ is the direct shift $(\gamma = 1)$ or $\alpha(t)$ is the inverse shift $(\gamma = -1)$.
If $\alpha(t)$ preserves orientation on $\Gamma (\gamma = 1)$, then
\begin{equation}\label{14}\begin{split}\det D\left( t \right) = \det \left( {P\left( t \right) - Q\left( t \right)} \right) = {c_1}\left( t \right){c_1}\left[ {\alpha \left( t \right)} \right] - {d_1}\left( t \right){d_1}\left[ {\alpha \left( t \right)} \right] = {\Delta _1}\left( t \right),\\
\det C\left( t \right) = \det \left( {P\left( t \right) + Q\left( t \right)} \right) = {a_1}\left( t \right){a_1}\left[ {\alpha \left( t \right)} \right] - {b_1}\left( t \right){b_1}\left[ {\alpha \left( t \right)} \right] = {\Delta _2}\left( t \right).
\end{split}\end{equation}
Similarly, if $\alpha(t)$ changes orientation on $\Gamma (\gamma=-1),$ then
\begin{equation}\label{15}\begin{split}\det D\left( t \right) =  - {b_1}\left( t \right){d_1}\left[ {\alpha \left( t \right)} \right] + {c_1}\left( t \right){a_1}\left[ {\alpha \left( t \right)} \right] = \Delta \left( t \right),\\
\det C\left( t \right) =  - {b_1}\left[ {\alpha \left( t \right)} \right]{d_1}\left( t \right) + {c_1}\left[ {\alpha \left( t \right)} \right]{a_1}\left( t \right) = \Delta \left[ {\alpha \left( t \right)} \right].
\end{split}\end{equation}

The principal result of this paper is the following which describes Noetherity of the operator $L.$
\begin{thm}\label{thm3}
In order for SIE (\ref{eq: 1}) with a Carleman shift to be Noetherian, it suffices to satisfy the following conditions
\begin{enumerate}[{\rm (i)}]
\item $\Delta_1 (t)\ne 0$, $\Delta_2 (t) \ne 0,$ if  $\alpha(t)$ is an orientation-preserving shift;
\item $\Delta(t) \ne 0,$  if $\alpha(t)$ is an orientation-changing shift.
\end{enumerate}
\end{thm}
\begin{proof} If condition (i) or (ii) holds, then $L$ is a Noetherian operator. Consequently, the numbers of linearly independent solutions of the homogeneous system of equations \eqref{eq: 5} and its union homogeneous system are finite, i.e. $l<\infty,$ $l^{*}<\infty.$ Therefore, it follows from Lemmas \ref{auxiliary lem1} and \ref{auxiliary lem2} that $l_1<\infty,$ $l_1^*<\infty.$ Now, we will see that
from the normally solvability of the system \eqref{eq: 5} follows normally solvability of the equation \eqref{eq: 1}. Indeed, for the solvability of the system of equations \eqref{eq: 1} it is necessary and sufficient to satisfy
\begin{equation}\label{a}
\int_{\Gamma}G(t)W(t)dt=0,
\end{equation}
where $G(t)=\{g(t),g[\alpha(t)]\}$ and $W(t)=\{w_1(t),w_2(t)\}$ is a solution of the system of equations \eqref{eq: 9} which is union with \eqref{eq: 5}.
Rewriting \eqref{a}, we obtain
\begin{equation}\label{b}
\int_{\Gamma}(g(t)w_1(t)+g[\alpha(t)]w_2(t))dt=0.
\end{equation}
 By Lemma \ref{auxiliary lem3}, \eqref{b} is a necessary and sufficient condition for solvability of the equation \eqref{eq: 1}.
 Let $W(t)$ be a solution of the union system \eqref{eq: 9} satisfying condition \eqref{l1*}
 $$w_2(t)-\gamma \alpha'(t)w_1(t)=0.$$
 Then, condition \eqref{b} is represented as
 $$\int_{\Gamma}g(t)w_1(t)dt+\gamma\int_{\Gamma}g[\alpha(t)]\alpha'(t)w_1[\alpha(t)]dt=0.$$
 By a substitution $t$ into $\alpha(t)$ in the previous equation, we obtain
  $$\int_{\Gamma}g(t)w_1(t)dt=0,$$
  where $w_1(t)=\psi(t)$ is any solution of the equation \eqref{eq: 7} which is union with \eqref{eq: 1}.
  Next, let $\psi(t)$ be a solution of the union system of equations \eqref{eq: 9} satisfying condition \eqref{l2*}
  $$w_2(t)+\gamma \alpha'(t) w_1[\alpha(t)]=0.$$
  In this case condition \eqref{b} is represented as
   $$\int_{\Gamma}g(t)w_1(t)dt-\gamma\int_{\Gamma}g[\alpha(t)]\alpha'(t)w_1[\alpha(t)]dt=0.$$
   Again, substituting $t$ into $\alpha(t)$ in the previous equation, we obtain that condition \eqref{b} satisfying \eqref{l2*} holds automatically for the solution of the union system \eqref{eq: 9}.
   Therefore, for the solvability of the equation \eqref{eq: 1} it is necessary and sufficient to satisfy the following condition
   $$\int_{\Gamma}g(t)\psi_{k}(t)dt=0, \quad k=1,2,...,l_1^*,$$
   where $\{\psi_k(t)\}$ ($k=1,2,...,l_1^*$) is a complete system of linearly independent solutions of the union equation \eqref{eq: 7}. This completes the proof.
   \end{proof}

Next result shows the Index formula for SIU (1) with a Carleman shift.
\begin{thm}\label{thm4}
If the Noetherian conditions are satisfied, then the index for SIU (1) with a Carleman shift is calculated by the following formulas:
\begin{enumerate}[{\rm (i)}]
\item If $\alpha(t)$ is an orientation-preserving shift, then
$$IndM = \frac{1}{{4\pi }}{\left\{ {arg\frac{{{\Delta _1}\left( t \right)}}{{{\Delta _2}\left( t \right)}}} \right\}_\Gamma};$$
\item  If $\alpha(t)$ is an orientation-changing shift, then
$$IndM = \frac{1}{{2\pi }}{\left\{ {arg\Delta \left( t \right)} \right\}_\Gamma }.$$
\end{enumerate}
\end{thm}
\begin{proof}By Lemmas \ref{auxiliary lem1} and \ref{auxiliary lem2}, we have
$$Ind L=Ind M+Ind K.$$
By definition, $Ind L=l-l^*,$ $Ind M= l_1-l_1^*,$ and $Ind K=l_2-l_2^*.$
By Remark \ref{indm=indk}, we have
$$Ind M=Ind K.$$
Consequently,
$$Ind L=2 Ind M.$$
Hence, the assertion follows from the formulas \eqref{eq: 13}-\eqref{15}.
\end{proof}

\section{Acknowledgment}
This work was partially supported by the grant No. AP05133283
of the Science Committee of the Ministry of Education and Science of the Republic of Kazakhstan.

\end{document}